\documentclass[11pt]{amsart}

\numberwithin{equation}{section}
\usepackage{hyperref, color}

\theoremstyle{plain}
\newtheorem{theorem}{Theorem}[section]

\newtheorem{lemma}[theorem]{Lemma}

\newtheorem{corollary}[theorem]{Corollary}

\theoremstyle{definition}

\newtheorem{definition}[theorem]{Definition}
\newtheorem{remark}[theorem]{Remark}

\newcommand{\C}{\mathbb{C}}
\newcommand{\CP}{\mathbb{C}P}
\newcommand{\Q}{\mathbb{Q}}

\title[Strong Cohomological rigidity of Bott manifolds]{Strong Cohomological rigidity of Bott manifolds}
\author[S. Choi]{Suyoung Choi}
\address{Department of mathematics, Ajou University, 206, World cup-ro, Yeongtong-gu, Suwon 16499,  Republic of Korea}
\email{schoi@ajou.ac.kr}
\author[T. Hwang]{Taekgyu Hwang}
\address{Department of mathematics, Ajou University, 206, World cup-ro, Yeongtong-gu, Suwon 16499, Republic of Korea}
\email{hwangtaekkyu@gmail.com}
\author[H. Jang]{Hyeontae Jang}
\address{Department of mathematics, Ajou University, 206, World cup-ro, Yeongtong-gu, Suwon 16499, Republic of Korea}
\email{a24325@ajou.ac.kr}

\thanks{The authors were supported by the National Research Foundation of Korea Grant funded by the Korean Government (NRF-2019R1A2C2010989).}

\date{\today}

\begin{document}
\begin{abstract}
We show that the strong cohomological rigidity conjecture for Bott manifolds is true. Namely, any graded cohomology ring isomorphism between two Bott manifolds is induced by a diffeomorphism.
\end{abstract}
\maketitle

\section{Introduction}
    A \emph{Bott tower} of height $n$ is a sequence of projective bundles
\begin{equation} \label{eq:BTower}
	B_n \xrightarrow{\pi_n} B_{n-1} \xrightarrow{\pi_{n-1}} \cdots \xrightarrow{\pi_3}
	B_2 \xrightarrow{\pi_2} B_1 \xrightarrow{\pi_1} B_0 = \{ \text{a point} \}.
\end{equation}
    Each $B_i$ is the projectivization $\mathbb{P}(\C \oplus \gamma^{\alpha_i})$ over $B_{i-1}$, where $\gamma^{\alpha_i}$ is the complex line bundle over $B_{i-1}$ whose first Chern class is $\alpha_i$. We call the total space~$B_n$ an $n$-stage \emph{Bott manifold}.

    When Grossberg and Karshon~\cite{Grossberg-Karshon1994} introduced the notion of a Bott tower, they described a Bott manifold as a special member in a one-parameter family of complex manifolds arising from a Bott--Samelson variety~\cite{Bott-Samelson1958}. Bott manifolds carry holomorphic actions of the algebraic torus of maximal dimension and form an important family of compact smooth toric varieties, which we call toric manifolds for short.

    One remarkable result in toric topology is that the classification of toric manifolds as varieties is closely related to the equivariant cohomology~\cite{Masuda2008Adv}.
    However, not much is known about the smooth classification.
    It was shown by Hirzebruch \cite{Hirzebruch1951} that two-stage Bott manifolds, known as \emph{Hirzebruch surfaces}, are diffeomorphic if and only if their cohomology rings are isomorphic as graded rings.
    In 2008, Masuda and Panov \cite{Masuda-Panov2008} showed that if the cohomology of $B_n$ is isomorphic to that of $(\CP^1)^n$, then $B_n$ is diffeomorphic to $(\CP^1)^n$.
    Motivated by this, the following smooth classification problem was posed in \cite{Masuda-Panov2008, Masuda-Suh2008} (see also \cite[Problem~7.8.29]{Buchstaber-Panov2015}) for the class of Bott manifolds, and more generally for the class of toric manifolds.

\medskip
\noindent
{\bf Cohomological rigidity problem for toric manifolds.}
Are toric manifolds diffeomorphic if their cohomology rings are isomorphic as graded rings?

\medskip
    There is no evidence so far that the problem gets easier when we instead consider the topological classification.
    Although the cohomology ring as an invariant is too weak to determine the diffeomorphism type in general, no counterexamples are known at the time of writing.

    By restricting the problem to the class of Bott manifolds, many partial answers were given; see \cite{Choi-Masuda-Suh2010TAMS}, \cite{Choi2015PEMS}, \cite{HK_arxiv}, \cite{Choi-Masuda-Murai2015}, \cite{Ishida2012}, \cite{Ishida2} for example.
    Interestingly, some results suggested that an even stronger proposition, called the \emph{strong cohomological rigidity conjecture}, should hold: any cohomology ring isomorphism between two Bott manifolds can be induced by a diffeomorphism.
    The conjecture was confirmed to be true for $n \leq 4$ so far, whereas it does not hold for general toric manifolds (see~\cite[Theorem~10]{Friedman-Morgan1988}).

	The purpose of the paper is to show that the strong cohomological rigidity conjecture is true for all Bott manifolds of arbitrary dimension $n$. Throughout the paper, we let $H^\ast(X)$ denote the integral cohomology ring of a topological space $X$.

\begin{theorem}[Strong cohomological rigidity for Bott manifolds]\label{thm:scr}
Let $B_n$ and $B_n'$ be Bott manifolds.
Given any graded ring isomorphism $\phi \colon H^\ast(B_n) \to H^\ast(B_n')$, there exists a diffeomorphism $\Phi \colon B_n' \to B_n$ such that $\Phi^\ast = \phi$.
\end{theorem}

	One of the key ingredients for the proof of Theorem~\ref{thm:scr} is Ishida's theorem (Theorem~\ref{thm:Ishida}) on the construction of the diffeomorphism. To apply his theorem, one needs to assume that $\phi$ is $k$-stable (see Definition~\ref{def:stable}) for $k=n-2$ or $k=n-1$. What we prove in this paper is that, by changing the tower structures of $B_n$ and $B'_n$ through diffeomorphisms, such assumption can always be made; see Theorem~\ref{thm:main} for the statement.

    Note that an $n$-stage Bott manifold is a toric manifold over an $n$-cube $I^n$.
    Hasui et al. \cite{HKMP2020} studied the cohomological rigidity for toric manifolds over an $n$-cube with one vertex cut, say $vc(I^n)$, assuming the cohomological rigidity for Bott manifolds.
    By Remark in \cite[p.4892]{HKMP2020}, we have the following non-trivial corollary.

    \begin{corollary}
        Let $X$ and $Y$ be toric manifolds over $vc(I^n)$.
        If $H^\ast(X) \cong H^\ast(Y)$ as graded rings, then $X$ and $Y$ are diffeomorphic.
    \end{corollary}

\section{Preliminaries}\label{sec:preliminary}
In this section, we fix notation that will be used for the rest of the paper and collect some results on the realization of an isomorphism between cohomology rings of Bott manifolds by a diffeomorphism.

\subsection{Notation}
Let $B_n$ be a Bott manifold defined in~\eqref{eq:BTower}. By Leray--Hirsch theorem, we have the following presentation of the cohomology ring
\begin{equation}\label{eq:coh}
	H^\ast(B_n) = \mathbb{Z}[x_1, \dots, x_n] / (x_i^2 - \alpha_i x_i)_{i=1}^n,
\end{equation}
where
\begin{equation}\label{eq:alpha}
	\alpha_i = \sum_{j=1}^{i-1} a_{ij} x_j, \quad a_{ij} \in \mathbb{Z}.
\end{equation}
Once this presentation is fixed, or equivalently the strictly lower triangular integer matrix $A = (a_{ij})_{n \times n}$ is given, we let $B_n(A)$ denote the Bott tower~\eqref{eq:BTower}. By slight abuse of the notation, the symbol $B_n(A)$ will also denote the Bott manifold $B_n$. Define the filtration
\begin{equation}\label{eq:filtration}
	F_k(A) := \mathrm{span} \{ x_1, \dots x_k \} \subset H^2(B_n(A))
\end{equation}
with the convention $F_0(A) := \{ 0 \}$. For given $\eta \in H^2(B_n(A))$, the \emph{height} of~$\eta$ is defined as
\begin{equation}
	\mathrm{ht}(\eta):= \min \{ k \mid \eta \in F_k(A) \}.
\end{equation}
\begin{remark}\label{rmk:zero}
	The expression~\eqref{eq:coh} shows that $\{x_i x_j\}_{i<j}$ forms an additive basis for $H^4(B_n(A))$. In particular, for $\eta_i \in H^2(B_n(A))$ with $\mathrm{ht}(\eta_i)<i$,
	\begin{equation}
		\sum_i \eta_i x_i = 0 \implies \eta_i =0 \quad \text{for all } i.
	\end{equation}
\end{remark}

Throughout the paper, we will use $A = (a_{ij})$ for the source and $B = (b_{ij})$ for the target. To avoid confusion, we will use the symbols $y_i$ and $\beta_i$ in place of $x_i$ and $\alpha_i$ for the matrix $B$.
\begin{definition}\label{def:stable}
	A graded ring isomorphism $\phi \colon H^\ast(B_n(A)) \to H^\ast(B_n(B))$ is called \emph{$k$-stable} if $\phi(F_k(A)) \subset F_k(B)$ for some $k$.
\end{definition}

\subsection{Realizable isomorphisms}
We say that an isomorphism between two cohomology rings is \emph{realizable} if it is induced from a diffeomorphism.

Let $B = (b_{ij})$ be an $n \times n$ matrix defining a Bott tower. In Section~\ref{sec:proof}, using the following two matrix operations, we will produce a new matrix $B' = (b'_{ij})$ that is better suited for the proof of Theorem~\ref{thm:main}. We provide the proofs for the reader's convenience. Following the expression~\eqref{eq:alpha}, we will use $\beta_i$ to express the $i$th row of $B$.

\begin{lemma}[Switching]\label{lem:switch}
	Suppose $b_{j+1, j} = 0$ for some $j$. Let $B'$ be the matrix obtained from $B$ by exchanging the $j$th and the $(j+1)$st rows and columns. Then the isomorphism $g \colon H^\ast(B_n(B)) \to H^\ast(B_n(B'))$ defined by
	\[
		g(y_j) = y'_{j+1}, \quad g(y_{j+1}) = y'_j, \quad \text{and} \quad g(y_i) = y'_i \text{ for } i \neq j, j+1
	\]
	is realizable.
\end{lemma}
\begin{proof}
	Since $b_{j+1,j} = 0$, the line bundle $\gamma^{\beta_{j+1}}$ can be regarded as a bundle over $B_{j-1}$. Then the space $B_{j+1}$ is the fiber product of $\mathbb{P}(\mathbb{C} \oplus \gamma^{\beta_j}) \to B_{j-1}$ and $\mathbb{P}(\mathbb{C} \oplus \gamma^{\beta_{j+1}}) \to B_{j-1}$. The diffeomorphism is given by switching the two factors of the fiber product.
\end{proof}

\begin{lemma}[Twisting]\label{lem:twist}
	Suppose there exists $v \in F_{j-1}(B)$ satisfying the equation $v(\beta_j - v) = 0$. Then there exists $B'$ such that
	\[
		\beta'_i = \beta_i \text{ for } i<j \quad \text{and} \quad \beta'_j = \beta_j - 2v
	\]
	and a realizable graded ring isomorphism $g \colon H^\ast(B_n(B)) \to H^\ast(B_n(B'))$ which is the identity when restricted to $F_{j-1}(B)$.
\end{lemma}
\begin{proof}
	Ishida~\cite[Theorem 3.1]{Ishida2012}	showed that rank $2$ decomposable vector bundles over a Bott manifold are isomorphic if and only if they have the same Chern classes. Since a projective bundle is unchanged under tensoring by a line bundle, we have the following isomorphisms of bundles:
	\[
		\mathbb{P}(\mathbb{C} \oplus \gamma^{\beta_j - 2v}) \cong \mathbb{P}((\mathbb{C} \oplus \gamma^{\beta_j - 2v}) \otimes \gamma^v) \cong \mathbb{P}(\gamma^v \oplus \gamma^{\beta_j - v}) \cong \mathbb{P}(\mathbb{C} \oplus \gamma^{\beta_j}).
	\]
	By pulling the bundle $B_n(B) \to \mathbb{P}(\mathbb{C} \oplus \gamma^{\beta_j})$ back by the above diffeomorphism, we obtain a new bundle $B_n(B') \to \mathbb{P}(\mathbb{C} \oplus \gamma^{\beta_j - 2v})$. The desired isomorphism $g$ is induced from the bundle isomorphism.
\end{proof}

Finally, we introduce the following theorem by Ishida that will be combined with Theorem~\ref{thm:main} to prove the strong cohomological rigidity.

\begin{theorem}[\cite{Ishida2012} for $k=n-1$ and \cite{Ishida2} for $k=n-2$]\label{thm:Ishida}
	Suppose that a graded ring isomorphism $\phi \colon H^\ast(B_n(A)) \to H^\ast(B_n(B))$ is $k$-stable and realizable. Then there exists a diffeomorphism $\Phi \colon B_n(B) \to B_n(A)$ such that $\Phi^\ast = \phi$.
\end{theorem}

\section{$\Q$-trivial decomposition}\label{sec:decomposition}
A Bott manifold is called \emph{$\Q$-trivial} if its rational cohomology ring is isomorphic to that of the product of the complex projective lines. Such Bott manifolds were introduced and classified in~\cite{Choi-Masuda2012}. In this section, we decompose general Bott towers into $\Q$-trivial ones and study how the decomposition plays a role in the cohomology ring.

\subsection{$\Q$-trivial Bott manifolds}
We begin with the lemma that classifies the square zero elements of degree~$2$ in the cohomology ring. Such elements generate the cohomology ring when the Bott manifold is $\Q$-trivial.
\begin{lemma}\label{lem:sz}
	Any element $z \in H^2(B_n(A))$ satisfying $z^2 = 0$ is of the form
	\[
		z = \lambda(2x_i - \alpha_i)
	\]
	with $\alpha_i^2 = 0$ for some $i$ and $\lambda \in \Q$.
\end{lemma}
\begin{proof}
	For any $z \neq 0$,	write $z = \lambda(2x_i + u)$ for some nonzero $\lambda \in \Q$, $i>0$ and $u \in H^2(B_n(A))$ with $\mathrm{ht}(u)<i$. Since $\mathrm{ht}(x_i) = i$, it follows from Remark~\ref{rmk:zero} and the expression~\eqref{eq:coh} that $u= -\alpha_i$ and $u^2=0$.
\end{proof}

Given an integer $j \geq 1$, let $\mathcal{H}_j$ denote the Bott tower determined by the following $j \times j$ matrix
\begin{equation}\label{eq:qmatrix}
	\begin{pmatrix}
		0 & 0 & \dots & 0 \\
		1 & 0 & \dots & 0 \\
		\vdots & \vdots & \ddots & \vdots\\
		1 & 0 & \dots & 0
	\end{pmatrix}.
\end{equation}
By the classification result of $\Q$-trivial Bott manifolds~\cite[Theorem~4.1]{Choi-Masuda2012},
\begin{equation}\label{eq:qtrivial}
	B_n(A) \cong \prod_{i=1}^t \mathcal{H}_{\lambda_i}
\end{equation}
for some partition $\lambda = (\lambda_1, \dots, \lambda_t)$ of $n$ when $B_n(A)$ is $\Q$-trivial.

\subsection{Towers of $\Q$-trivial manifolds}

Given a matrix $A$ defining a Bott tower, one can construct a new one by taking a submatrix of A.
For any integer $1 \leq k < n$, let $\hat{A}$ and $\bar{A}$ denote the $k \times k$ upper left matrix and the $(n-k) \times (n-k)$ lower right matrix, respectively.
Explicitly,
\begin{equation}
	\hat{A}_{ij}:= a_{ij} \quad \text{and} \quad \bar{A}_{ij}:= a_{i+k, j+k}.
\end{equation}
The Bott manifold $B_n(A)$ can be geometrically interpreted as a fiber bundle over $B_k(\hat{A})$ with fiber $B_{n-k}(\bar{A})$.
The definition of $\hat{A}$ and $\bar{A}$ depends on the choice of $k$, but we will omit it from the notation when it is clear from the context.

By applying the switching operations (Lemma~\ref{lem:switch}) if necessary, we may, hence will always, assume that a Bott tower $B_n(A)$ is well-ordered (see \cite[Lemma~2.3]{Choi2015PEMS}).
That is,
\[
	\alpha_j^2=0 \implies \alpha_i^2=0 \quad \text{for } i<j.
\]
Let $k:= \max \{i \mid \alpha_i^2 =0 \}$. Then the space $B_n(A)$ has a fibration structure
\begin{equation}
	B_{n-k}(\bar{A}) \to B_n(A) \to B_k(\hat{A}).
\end{equation}
Since $ 2x_i - \alpha_i$ for $i=1, \dots, k$ generates $H^\ast(B_k(\hat{A}), \Q)$, the base space $B_k(\hat{A})$ is $\Q$-trivial.
By repeating this process for $B_{n-k}(\bar{A})$, we obtain a tower
\begin{equation}
	Q_m  \xrightarrow{p_m} Q_{m-1}  \xrightarrow{p_{m-1}} \dots  \xrightarrow{p_2} Q_1 = B_k(\hat{A}) \xrightarrow{p_1} Q_0 = \{ \text{a point} \}
\end{equation}
whose fibers are $\Q$-trivial Bott manifolds. Since each $p_i$ induces an injection on the cohomology ring, we will regard $H^2(Q_i)$ as a subgroup of $H^2(Q_m)$.
\begin{definition}
	The \emph{level} of $\eta \in H^2(B_n(A))$ is defined to be
	\begin{equation}
		\mathrm{lev}(\eta):= \min \{ j \mid \eta \in H^2(Q_j) \}.
	\end{equation}
\end{definition}
By the above construction of $\Q$-trivial towers, we have
\begin{equation}\label{eq:well-ordered}
	\mathrm{ht}(\eta) \leq \mathrm{ht}(\xi) \implies \mathrm{lev}(\eta) \leq \mathrm{lev}(\xi).
\end{equation}
We will always assume this throughout the paper. Since a graded ring isomorphism preserves the set of square zero elements, the following lemma follows from induction (cf. \cite[Proposition~4.1]{Choi-Masuda-Murai2015}).
\begin{lemma}\label{lem:stable}
	Let $\phi \colon H^\ast(B_n(A)) \to H^\ast(B_n(B))$ be a graded ring isomorphism. Then there exists a permutation $\sigma \colon [n] \to [n]$ such that
	\begin{equation}\label{eq:permutation}
		\phi(2x_i - \alpha_i) = \epsilon_i (2y_{\sigma(i)} - \beta_{\sigma(i)}) \quad \text{and} \quad \mathrm{lev}(x_i) = \mathrm{lev}(y_{\sigma(i)})
	\end{equation}
	for some $\epsilon_i \in \Q$ and $i=1, \dots, n$. In particular, the isomorphism $\phi$ is $k$-stable if $k = \dim_\C Q_j$ for some $j = 0, \dots, m$.
\end{lemma}

\subsection{Blocks}
Let $B_n(A)$ be a Bott tower well-ordered in the sense of~\eqref{eq:well-ordered}. Pick $i$ and $j$ so that $x_i$ and $x_j$ have the same level, say $\ell$. The quotient $H^\ast(Q_{\ell})/H^\ast(Q_{\ell-1})$ is isomorphic to the cohomology ring of a $\Q$-trivial Bott manifold. Fix a graded ring isomorphism
\begin{equation}\label{eq:psi}
	\psi \colon H^\ast(Q_{\ell})/H^\ast(Q_{\ell-1}) \to \prod_i H^\ast(\mathcal{H}_{\lambda_i}).
\end{equation}
\begin{definition}
	Two numbers $i$ and $j$ are said to be \emph{in the same block for~$A$} if $\mathrm{lev}(x_i) = \mathrm{lev}(x_j)$ and there exists $k$ such that both $\psi(2x_i - \alpha_i)$ and $\psi(2x_j - \alpha_j)$ are elements of $H^2(\mathcal{H}_k)$.
\end{definition}
We remark that this definition is independent of the choice of $\psi$ because any graded automorphism of $\prod_i H^\ast(\mathcal{H}_{\lambda_i})$ must send $H^2(\mathcal{H}_{\lambda_j})$ isomorphically onto $H^2(\mathcal{H}_{\lambda_{j'}})$ for some $j'$.

\begin{lemma}\label{lem:permutation}
	Let $\phi \colon H^\ast(B_n(A)) \to H^\ast(B_n(B))$ be a graded ring isomorphism. Let $\sigma \colon [n] \to [n]$ be the permutation defined by the equation~\eqref{eq:permutation}. Then $i$ and $j$ are in the same block for $A$ if and only if $\sigma(i)$ and $\sigma(j)$ are in the same block for $B$.
\end{lemma}

Next, we characterize the condition when $i$ and $j$ are in the same block for $A$ in terms of the entries of the matrix. Let $k = \dim_\C Q_{\ell-1}$ for some $\ell \geq 1$. For each $r>k$, let $z_r$ denote the image of $2x_r - \alpha_r$ under the projection
\begin{equation}\label{eq:projection}
	H^2(B_n(A)) \to H^2(B_{n-k}(\bar{A})),
\end{equation}
divided by $2$ when it is not primitive. That is, $z_r$ is primitive and $z_r^2=0$ when $\mathrm{lev}(x_r) = \ell$.
\begin{lemma}\label{lem:block}
	Pick $i,j>k$ so that $\mathrm{lev}(x_i) = \mathrm{lev}(x_j) = \ell$. Then $i$ and $j$ are in the same block for $A$ if and only if $z_i \equiv z_j \mod 2$.
\end{lemma}
\begin{proof}
	By Lemma~\ref{lem:sz}, expression~\eqref{eq:qmatrix} and \eqref{eq:qtrivial}, the image of a primitive square zero element under~\eqref{eq:psi} must be of the form $\pm x_r$ or $\pm(2x_s - x_r)$.
\end{proof}

\begin{corollary}\label{cor:block}
	Let $A$ be a matrix defining a Bott tower with $a_{j+1,j}$ odd.
	\begin{enumerate}
		\item
			If $\mathrm{lev}(x_{j}) = \mathrm{lev}(x_{j+1})$, then $j$ and $j+1$ are in the same block.
		\item
			$i$ and $j+1$ are not in the same block for any $i<j$.
	\end{enumerate}
\end{corollary}
\begin{proof}
	Let $\ell = \mathrm{lev}(x_j)$. Let $\bar{x}_r$ and $\bar{\alpha}_r$ denote the image of $x_r$ and $\alpha_r$ under~\eqref{eq:projection}. By Lemma~\ref{lem:sz}, we have $\bar{\alpha}_{j+1} = a_{j+1,j}(\bar{x}_j - \bar{\alpha}_j/2)$. This implies that $\bar{\alpha}_j$ is even. Since $z_j = \bar{x}_j - \bar{\alpha}_j/2$ and $z_{j+1} = 2\bar{x}_{j+1} - \bar{\alpha}_{j+1}$, the first statement follows from Lemma~\ref{lem:block}. The second statement also follows from Lemma~\ref{lem:block}.
\end{proof}

\section{Proof}\label{sec:proof}

The following is the key lemma to prove our main theorem.
\begin{lemma}\label{lem:key}
	Let $\phi \colon H^\ast(B_n(A)) \to H^\ast(B_n(B))$ be a $k$-stable isomorphism. Let $\ell > k+1$ be the smallest integer satisfying the condition $\phi(x_{k+1}) \in F_{\ell}(B)$. Set $p:= b_{\ell, \ell-1}$. Depending on the parity of $p$, we can find an $n \times n$ matrix~$B'$ and a $k$-stable realizable graded isomorphism $g \colon H^\ast(B_n(B)) \to H^\ast(B_n(B'))$ satisfying the following properties.
	\begin{enumerate}
		\item \label{even}
			If $p$ is even, then $(g \circ \phi) (x_{k+1}) \in F_{\ell - 1}(B')$. The $i$th row of $B'$ coincides with that of $B$ for $i < \ell - 1$.
		\item \label{odd}
			If $p$ is odd, then $(g \circ \phi) (x_{k+1}) \in F_{\ell - 1}(B')$ provided that $\ell > k+2$. The $i$th row of $B'$ coincides with that of $B$ for $i < \ell - 2$.
	\end{enumerate}
\end{lemma}
\begin{proof}
	We will derive properties that hold regardless of the parity of $p$ and divide the cases later when needed. For each $i > k$, let
	\begin{equation}\label{eq:bar}
		\bar{\beta}_i:= \sum_{j=k+1}^{i-1} b_{ij} y_j \in H^2(B_n(B))
	\end{equation}
	denote the class obtained from $\beta_i$ by truncating the terms in $F_{k}(B)$. Since $\phi$ is $k$-stable, the image of $\alpha_k$ under $\phi$ is contained in $F_{k}(B)$. Using Lemma~\ref{lem:stable} and the assumption that $\ell$ is the smallest integer satisfying the condition $\phi(x_{k+1}) \in F_{\ell}(B)$, we may write
	\begin{equation}\label{eq:xk}
		\phi(x_{k+1}) = \epsilon (2y_{\ell} - \bar{\beta}_{\ell}) + w
	\end{equation}
	for some nonzero $\epsilon \in \Q$ and $w \in F_{k}(B)$. If $p=0$, the proof is complete by applying Lemma~\ref{lem:switch} for $j=\ell-1$.
	
	We assume $p \neq 0$ from now on. If follows from~\eqref{eq:coh} and~\eqref{eq:xk} that
	\begin{align*}
		0 =& \phi \left(x_{k+1}(x_{k+1} - \alpha_{k+1})\right) = (2\epsilon y_{\ell} - \epsilon \bar{\beta}_{\ell} + w)(2\epsilon y_{\ell} - \epsilon \bar{\beta}_{\ell} + w - \phi(\alpha_{k+1}))\\
		=& 2\epsilon y_{\ell}(2\epsilon y_{\ell} - 2\epsilon \bar{\beta}_{\ell} + 2w - \phi(\alpha_{k+1}))\\
		&+ \epsilon \bar{\beta}_{\ell}(\epsilon \bar{\beta}_{\ell} -2w + \phi(\alpha_{k+1}))\\
		&+ w(w-\phi(\alpha_{k+1})).
	\end{align*}
	We can remove all the square terms using the relation $y_i^2 = \beta_i y_i$. From the expression~\eqref{eq:bar} and the condition that $w, \phi(\alpha_{k+1}) \in F_k(B)$, all the three terms in the last expression are zero by Remark~\ref{rmk:zero}. By setting
	\begin{equation}
		u:= \frac{-2w + \phi(\alpha_{k+1})}{\epsilon},
	\end{equation}
	the first two terms lead to the relations
	\begin{equation}\label{eq:1}
		\beta_{\ell} = \bar{\beta}_{\ell} + \frac{-2w + \phi(\alpha_{k+1})}{2\epsilon} = \bar{\beta}_{\ell} + \frac{u}{2}
	\end{equation}
	and
	\begin{equation}\label{eq:2}
		\bar{\beta}_{\ell} \left(\bar{\beta}_{\ell} + u\right) = \bar{\beta}_{\ell} \left(\bar{\beta}_{\ell} + \frac{-2w + \phi(\alpha_{k+1})}{\epsilon}\right) = 0.
	\end{equation}
	Since $x_{k+1}^2$ is contained in the ideal generated by $F_{k}(A)$ in $H^\ast(B_n(A))$, using the equation~\eqref{eq:xk} and the assumption that $\phi$ is $k$-stable, we see that $(2y_{\ell} - \bar{\beta}_{\ell})^2$ is contained in the ideal generated by $F_{k}(B)$. By the classification of the square zero elements (Lemma~\ref{lem:sz}), we see that
	\begin{equation}\label{eq:3}
		\bar{\beta}_{\ell} = p\left(y_{\ell-1} - \frac{\bar{\beta}_{\ell-1}}{2}\right).
	\end{equation}
	We divide the cases by the parity of $p$.

\medskip
	\noindent\textsf{\textbf{Case 1}: $p$ is even.}

	By \eqref{eq:1} and \eqref{eq:2}, we have
	\begin{equation}\label{eq:even}
		\bar{\beta}_{\ell}(2\beta_{\ell} - \bar{\beta}_{\ell}) = 0.
	\end{equation}
	Substituting \eqref{eq:3} into \eqref{eq:even}, it follows that
	\begin{align*}
		0 &= \left(y_{\ell-1} - \frac{\bar{\beta}_{\ell-1}}{2}\right)\left(2\beta_{\ell} - py_{\ell-1} + \frac{p}{2}\bar{\beta}_{\ell-1} \right)\\
			&= y_{\ell-1} (2\beta_{\ell} - py_{\ell-1} ) - \frac{\bar{\beta}_{\ell-1}}{2}\left( 2\beta_{\ell} - 2py_{\ell-1} + \frac{p}{2}\bar{\beta}_{\ell-1} \right)\\
			&= 2y_{\ell-1} \left(\beta_{\ell} - \frac{p}{2}y_{\ell-1} \right) - \frac{\bar{\beta}_{\ell-1}}{2}\left( u - \frac{p}{2}\bar{\beta}_{\ell-1} \right),
	\end{align*}
	where the last equality follows from \eqref{eq:1} and \eqref{eq:3}. Since $\bar{\beta}_{\ell-1}$ has no $y_{\ell-1}$ term and $u \in F_{k}(B)$, both terms in the last expression vanish. Hence,
	\[
		\frac{p}{2}y_{\ell-1} \left(\beta_{\ell} - \frac{p}{2}y_{\ell-1} \right) = 0.
	\]
	By applying Lemma~\ref{lem:twist} for $j=\ell$ and $v = py_{\ell-1}/2$, we can find a new matrix $B' = (b'_{ij})$ with $b'_{\ell, \ell-1} = 0$. This reduces the proof to the case $p=0$.

\medskip
\noindent\textsf{\textbf{Case 2}: $p$ is odd.}

	Substituting \eqref{eq:3} into \eqref{eq:2}, we get
	\begin{align*}
		0 &= \left(y_{\ell-1} - \frac{\bar{\beta}_{\ell-1}}{2}\right)\left(y_{\ell-1} - \frac{\bar{\beta}_{\ell-1}}{2} + \frac{u}{p}\right)\\
			&= y_{\ell-1}\left(y_{\ell-1} - \bar{\beta}_{\ell-1} + \frac{u}{p}\right) + \frac{\bar{\beta}_{\ell-1}}{2}\left(\frac{\bar{\beta}_{\ell-1}}{2} - \frac{u}{p} \right).
	\end{align*}
	By the similar reasoning as before, it follows that
	\begin{equation}\label{eq:4}
		\beta_{\ell-1} = \bar{\beta}_{\ell-1} - \frac{u}{p}
	\end{equation}
	and
	\begin{equation}\label{eq:5}
		\frac{\bar{\beta}_{\ell-1}}{2}\left(\frac{\bar{\beta}_{\ell-1}}{2} - \frac{u}{p} \right) = 0.
	\end{equation}
	Since $p$ is odd, the class $\bar{\beta}_{\ell-1}$ is divisible by $2$ by \eqref{eq:3}. By taking $v:= \bar{\beta}_{\ell-1}/2$, we obtain $v(\beta_{\ell-1} - v) = 0$ from \eqref{eq:4} and \eqref{eq:5}. By Lemma~\ref{lem:twist}, we obtain a new matrix $B' = (b'_{ij})$ with
	\begin{equation}\label{eq:zero}
		b'_{\ell-1, k+1} = b'_{\ell-1, k+2} = \dots = b'_{\ell-1, \ell-2} = 0.
	\end{equation}
	The assumption that $\ell > k+2$ implies that \eqref{eq:zero} is not vacuous. Note that $b'_{\ell, \ell - 2}$ must be zero since $(\bar{\beta}'_{\ell})^2$ is contained in the ideal generated by $F_k(B')$.
	Applying Lemma~\ref{lem:switch} for $j= \ell - 2$, we obtain a new matrix $B'' = (b''_{ij})$ satisfying
	\begin{equation}
		b''_{\ell, \ell - 1} = b'_{\ell, \ell - 2}=0.
	\end{equation}
	This reduces the proof to the case $p=0$ and the proof is complete. Note that we used Lemma~\ref{lem:switch} for $j=\ell-2$ and $\ell-1$, and Lemma~\ref{lem:twist} for $j=\ell-1$, so the $i$th row for $i<\ell-2$ is unchanged.
\end{proof}

\begin{theorem}\label{thm:main}
	Let $k<n$ be the largest integer such that the graded ring isomorphism $\phi \colon H^\ast(B_n(A)) \to H^\ast(B_n(B))$ is $k$-stable. Then there exist $n \times n$ matrices $A'$ and $B'$, and realizable graded ring isomorphisms
	\[
		f \colon H^\ast(B_n(A')) \to H^\ast(B_n(A)) \quad \text{and} \quad g \colon H^\ast(B_n(B)) \to H^\ast(B_n(B'))
	\]
	such that the composition $\phi':= g \circ \phi \circ f$ is either $(k+1)$ or $(k+2)$-stable. In particular, if $\phi$ is realizable, so is $\phi'$.
\end{theorem}

\begin{proof}
	Pick smallest integer $\ell$ such that $\phi(x_{k+1}) \in F_{\ell}(B)$. If $\ell = k+1$, there is nothing to prove.
	
	For $\ell > k+1$, by applying Lemma~\ref{lem:key} repeatedly, we can find $B'$ and $g \colon H^\ast(B_n(B)) \to H^\ast(B_n(B'))$ satisfying
	\begin{equation}
		(g \circ \phi)(x_{k+1}) \in F_{k+2}(B').
	\end{equation}
	If $b'_{k+2,k+1}$ is even, the proof is complete by Lemma~\ref{lem:key}.
	
	Now assume that $b'_{k+2,k+1}$ is odd. From the choice of $k$ and Lemma~\ref{lem:stable}, it follows that $\mathrm{lev}(y'_{k+1}) = \mathrm{lev}(y'_{k+2})$. By Corollary~\ref{cor:block}, we see that $k+1$ and $k+2$ are in the same block for $B'$.
	
	Next, we apply Lemma~\ref{lem:key} using
	\[
		(g \circ \phi)^{-1} \colon H^\ast(B_n(B')) \to H^\ast(B_n(A))
	\]
	while keeping the $(k+1)$st row of $A$ unchanged. Then we find $A'$ and $f^{-1} \colon H^\ast(B_n(A)) \to H^\ast(B_n(A'))$ such that
	\begin{equation}
		(g \circ \phi \circ f)^{-1}(y'_{k+1}) \in F_{k+3}(A') \quad \text{and} \quad (g \circ \phi \circ f)^{-1}(y'_{k+2}) \in F_{k+1}(A').
	\end{equation}
	If the image of $y'_{k+1}$ has height $k+2$, the proof is complete. So we assume the height is $k+3$. Since $k+1$ and $k+2$ are in the same block for $B'$, Lemma~\ref{lem:permutation} shows that $k+1$ and $k+3$ are in the same block for $A'$. By Corollary~\ref{cor:block}, the entry $a'_{k+3, k+2}$ must be even. By applying Lemma~\ref{lem:key} once again, we are done.
\end{proof}

\begin{theorem}[Strong cohomological rigidity for Bott manifolds]
Given any graded ring isomorphism $\phi \colon H^\ast(B_n(A)) \to H^\ast(B_n(B))$, there exists a diffeomorphism $\Phi \colon B_n(B) \to B_n(A)$ such that $\Phi^\ast = \phi$.
\end{theorem}
\begin{proof}
	The proof is by induction on $n$. Assume that the statement holds for all Bott towers of height less than $n$. Let $\phi \colon H^\ast(B_n(A)) \to H^\ast(B_n(B))$ be a graded ring isomorphism. Take the largest $k<n$ such that $\phi$ is $k$-stable and let $\phi_k$ denote the restriction of $\phi$. By applying Theorem~\ref{thm:main} repeatedly and replacing $A$, $B$, $\phi$ with $A'$, $B'$, $\phi'$, we may assume that $k=n-2$ or $k=n-1$. Now the proof is complete by Theorem~\ref{thm:Ishida}.
\end{proof}

\providecommand{\bysame}{\leavevmode\hbox to3em{\hrulefill}\thinspace}
\providecommand{\MR}{\relax\ifhmode\unskip\space\fi MR }
\providecommand{\MRhref}[2]{%
  \href{http://www.ams.org/mathscinet-getitem?mr=#1}{#2}
}
\providecommand{\href}[2]{#2}


\begin{thebibliography}{10}

\bibitem{Bott-Samelson1958}
Raoul Bott and Hans Samelson, \emph{Applications of the theory of {M}orse to
  symmetric spaces}, Amer. J. Math. \textbf{80} (1958), 964--1029. \MR{105694}

\bibitem{Buchstaber-Panov2015}
Victor~M. Buchstaber and Taras~E. Panov, \emph{Toric topology}, Mathematical
  Surveys and Monographs, vol. 204, American Mathematical Society, Providence,
  RI, 2015. \MR{3363157}

\bibitem{Choi2015PEMS}
Suyoung Choi, \emph{Classification of {B}ott manifolds up to dimension 8},
  Proc. Edinb. Math. Soc. (2) \textbf{58} (2015), no.~3, 653--659. \MR{3391366}

\bibitem{Choi-Masuda2012}
Suyoung Choi and Mikiya Masuda, \emph{Classification of {$\Bbb Q$}-trivial
  {B}ott manifolds}, J. Symplectic Geom. \textbf{10} (2012), no.~3, 447--461.
  \MR{2983437}

\bibitem{Choi-Masuda-Murai2015}
Suyoung Choi, Mikiya Masuda, and Satoshi Murai, \emph{Invariance of
  {P}ontrjagin classes for {B}ott manifolds}, Algebr. Geom. Topol. \textbf{15}
  (2015), no.~2, 965--986. \MR{3342682}

\bibitem{Choi-Masuda-Suh2010TAMS}
Suyoung Choi, Mikiya Masuda, and Dong~Youp Suh, \emph{Topological
  classification of generalized {B}ott towers}, Trans. Amer. Math. Soc.
  \textbf{362} (2010), no.~2, 1097--1112. \MR{2551516}

\bibitem{Friedman-Morgan1988}
Robert Friedman and John~W. Morgan, \emph{On the diffeomorphism types of
  certain algebraic surfaces. {I}}, J. Differential Geom. \textbf{27} (1988),
  no.~2, 297--369. \MR{925124}

\bibitem{Grossberg-Karshon1994}
Michael Grossberg and Yael Karshon, \emph{Bott towers, complete integrability,
  and the extended character of representations}, Duke Math. J. \textbf{76}
  (1994), no.~1, 23--58. \MR{1301185}

\bibitem{HKMP2020}
Sho Hasui, Hideya Kuwata, Mikiya Masuda, and Seonjeong Park,
  \emph{Classification of toric manifolds over an {$n$}-cube with one vertex
  cut}, Int. Math. Res. Not. IMRN (2020), no.~16, 4890--4941. \MR{4139029}

\bibitem{HK_arxiv}
Akihiro Higashitani and Kazuki Kurimoto, \emph{Cohomological rigidity for
  {F}ano {B}ott manifolds}, Math. Z., to appear, arXiv:2008.05811.

\bibitem{Hirzebruch1951}
Friedrich Hirzebruch, \emph{\"{U}ber eine {K}lasse von
  einfachzusammenh\"{a}ngenden komplexen {M}annigfaltigkeiten}, Math. Ann.
  \textbf{124} (1951), 77--86. \MR{45384}

\bibitem{Ishida2012}
Hiroaki Ishida, \emph{Filtered cohomological rigidity of {B}ott towers}, Osaka
  J. Math. \textbf{49} (2012), no.~2, 515--522. \MR{2945760}

\bibitem{Ishida2}
\bysame, \emph{Strong cohomological rigidity of {H}irzebruch surface bundles in
  {B}ott towers}, arXiv:2111.07299, 2021.

\bibitem{Masuda-Panov2008}
M.~Masuda and T.~E. Panov, \emph{Semi-free circle actions, {B}ott towers, and
  quasitoric manifolds}, Mat. Sb. \textbf{199} (2008), no.~8, 95--122.
  \MR{2452268}

\bibitem{Masuda2008Adv}
Mikiya Masuda, \emph{Equivariant cohomology distinguishes toric manifolds},
  Adv. Math. \textbf{218} (2008), no.~6, 2005--2012. \MR{2431667}

\bibitem{Masuda-Suh2008}
Mikiya Masuda and Dong~Youp Suh, \emph{Classification problems of toric
  manifolds via topology}, Toric topology, Contemp. Math., vol. 460, Amer.
  Math. Soc., Providence, RI, 2008, pp.~273--286. \MR{2428362}

\end{thebibliography}
\end{document}